\documentclass[dvipsnames]{siamltex}
\usepackage{amssymb,amsmath}
\usepackage{color}
\usepackage{paralist}
\usepackage{graphics} 
\usepackage{epsfig} 
\usepackage[pagewise]{lineno}
\usepackage{graphicx}  
\usepackage{epstopdf}
\usepackage[colorlinks=true]{hyperref}
\hypersetup{urlcolor=blue, citecolor=red}

\usepackage{hyperref}
\usepackage{xspace}

\newtheorem{remark}[theorem]{Remark}
\newtheorem{example}[theorem]{Example}
\newtheorem{assumption}[theorem]{Assumption}

\newcommand{\raumH}{H^1_{[0]}(-b,T;\mathbb{R}^n)}
\newcommand{\raumL}{L^2(-b,T;\mathbb{R}^n)}
\newcommand{\raumHeins}{H^1(-b,T;\mathbb{R}^n)}
\newcommand{\Lnull}{L^2_0(-b,T;\mathbb{R}^n)}

\newcommand{\fredi}{\textcolor{cyan}}

\title{Second order analysis for the optimal selection of time delays
}
\author{Karl Kunisch\thanks{University of Graz, Institute of Mathematics, Scientific Computing, Graz, Austria 
 \ and \ Radon Institute, Austrian Institute of Sciences, Linz, Austria (karl.kunisch@uni-graz.de).}
\and Fredi Tr\"oltzsch\thanks{Institut f\"ur Mathematik,
Technische Universit\"at Berlin, D-10623 Berlin, Germany
(troeltzsch@math.tu-berlin.de).}}

\begin{document}

\maketitle

\begin{abstract}
For a nonlinear ordinary differential equation with time delay, the differentiation of the solution with respect to
the delay is investigated. Special emphasis is laid on the second-order derivative. The results are applied to an associated optimization problem for the time delay. A first- and second-order sensitivity analysis is performed including an adjoint calculus that avoids the second derivative of the state with respect to the delay. 
\end{abstract}



\pagestyle{myheadings} \thispagestyle{plain} \markboth{K. KUNISCH
AND F. TR\"OLTZSCH}{Optimization of time delays}

\section{Introduction} \label{S1}
In this paper, we discuss the differentiability of the solution of the delay differential equation
\begin{equation}
\begin{array}{rcll}
\dot x(t) + f(x(t))&=& A \, x(t-\tau) + g(t) & \mbox{ in } (0,T),\\
x(t) &=& \varphi(t) & \mbox{ in }[-b,0]
\end{array} \label{delayeq}
\end{equation}
with respect to the time delay $\tau$. More precisely, denoting the solution of this equation by $x[\tau]$,
we show the existence  of the first- and second-order derivatives of the mapping $\tau \mapsto x[\tau]$ and derive equations for them. 

In \eqref{delayeq}, the following quantities are given:  A continuously differentiable 
function $f: \mathbb{R}^n \to \mathbb{R}^n$, a matrix  $A \in \mathbb{R}^{n\times n}$, a time delay $\tau \ge 0$,
a fixed terminal time $T > 0$, and functions $g: [0,T] \to \mathbb{R}^n, \, \varphi: [-b,0] \to \mathbb{R}^n$.
Here, $b > 0$ is a fixed bound such that $\tau$ can vary in the interval $[0,b]$. 

As an application of the differentiability properties of the mapping $\tau \mapsto x[\tau]$, we derive first- and second-order optimality conditions for the following delay optimization problem: 
\begin{equation}\label{optim}
\min_{0 \le \tau \le b} \int_0^T |x[\tau](t) - x_d(t)|^2 dt,
\end{equation}
where $x_d \in L^2(0,T;\mathbb{R}^n)$ is a given desired state.

Our paper contributes to the control theory of delay equations that is a well developed field of applied mathematics. 
Among the very many contribution we can only cite a very small selection from the distant \cite{bb79,halanay68,hale1971,kk82,kaku87} and more recent past \cite{breda23, Michiels2013}. 

In theoretical physics, stability properties and the control of systems of delay equation became an important issue. There is an active research in feedback control and stabilization of chaotic systems. We refer to the seminal paper \cite{pyragas1992}, to \cite{Erneux2009},  and to the survey
\cite{schoell_schuster2008} with various applications. We mention exemplarily the design of lasers or the research on neurological diseases. The dependence of solutions on the delays is an interesting and significant question. In particular, this concerns the differentiability with respect to delays. In \cite{hale_ladeira1991}, higher order differentiability was shown for a nonlinear differential equation with delay,
in  \cite{hale_ladeira1993} for a class of nonlinear retarded reaction diffusion equations. Both results
were proved only locally in time. 

Recently, in \cite{casas_mateos_troeltzsch2019} the optimization of time delays in semilinear parabolic partial differential equation was investigated in the context of optimal control theory. The results were
based on a general theory of first-order necessary optimality conditions  for optimal control problems with nonlocal measure control of parabolic equations, \cite{casas_mateos_troeltzsch2018}. An optimization problem of  feedback controllers for a parabolic equation with nonlocal time delay was discussed in \cite{Nestler_etal2016}. We also mention \cite{casas_yong2023}, where a nonlocal optimal control problem with memory and measure-valued controls is considered. All the results cited in this block are global in time.

The main novelty of our paper is the second-order 
sensitivity analysis for the optimization of the time delay in a nonlinear system of 
delay differential equations. In particular, we prove the first- and second-order differentiability of the state w.r. to the delay. Moreover, we present the sensitivity analysis for the optimization problem \eqref{optim} -- first by adjoint calculus without invoking the second derivative of the state w.r. to $\tau$ and later on using this second-order derivative. We improve the results of \cite{hale_ladeira1991}, \cite{hale_ladeira1993}, where a sufficiently small time horizon is  assumed for the differentiability results. We are able to derive results that are global in time. 

The paper is organized as follows: In Section \ref{S2}  well-posedness of equation \eqref{delayeq} is proven
and the regularity of its solution is discussed. Section \ref{S3} is devoted to the differentiability of the state $x$ with respect to the delay $\tau$. The  first and second-order  sensitivity analysis of the optimization problem is addressed in Section \ref{S4} via an adjoint calculus without using the second-order derivative of the state with respect  to $\tau$.
The (global) second-order differentiability of the state with respect  to $\tau$ is the topic of Section \ref{S5}.  
Section \ref{S6} contains a brief discussion of the case of multiple time delays.

\section{The delay differential equation} \label{S2}

The aim of this section consists in establishing existence and uniqueness of a solution $x$ to \eqref{delayeq}. Throughout the paper, we will require the following standing assumptions on $f$, $\varphi$, and $g$, where $Df(x)$ will denote the Jacobian matrix of $f$ at $x$ and $I$ is the identity matrix. 

\begin{assumption} \label{A1}The function $f$ is continuously differentiable and there is a constant $\lambda > 0$, such that  
\begin{equation} \label{lowerbound}
Df(x) + \lambda I \text{ is positive semi-definite for all }  x \in  \mathbb{R}^n.
\end{equation}
The function $g$ belongs to $L^2(0,T; \mathbb{R}^n)$ and $\varphi$ to $\raumHeins$.
\end{assumption}

Later, in the context of differentiability, we will slightly strengthen the assumptions on $f, \, \varphi$, and $g$.
For $n=1$, a typical non-monotone candidate is $f(x) =  (x-x_1)(x-x_2)(x-x_3)$ with given real numbers $x_1 \le x_2 \le x_3$.

Prior to the discussion of equation \eqref{delayeq}, we first consider the auxiliary system
\begin{equation}
\begin{array}{l}
\dot x(t) + f(x(t))= g(t) \quad \mbox{ in } (0,T),\\
x(0) = x_0,
\end{array} \label{aux1}
\end{equation}
where $x_0 \in  \mathbb{R}^n$ is given. 

A function $x \in H^1(0,T;\mathbb{R}^n)$ is said to be a solution of \eqref{aux1}, if it satisfies 
the equation almost everywhere in $(0,T)$ and obeys the initial condition. 
If $g$ is continuous, then we can consider $x$ as classical solution, i.e. $x \in C^1([0,T],\mathbb{R}^n).$

\begin{proposition}\label{L1}Assume that $f$ and $g$ obey Assumption \ref{A1}. 
Then, for all $x_0 \in \mathbb{R}^n$, equation \eqref{aux1}
has a unique solution $x \in H^1(0,T;\mathbb{R}^n)$. It satisfies 
\[
\|x\|_{H^1(0,T;\mathbb{R}^n)} \le C(|x_0|, |f(0)|, T, \|g\|_{L^{2}(0,T;\mathbb{R}^n)} ), 
\]
where $C$ is continuous and monotonically increasing in each of its arguments. 
\end{proposition}
\begin{proof}
(i) Utilizing the  {\em{ transformation}} $x(t) = e^{\lambda t} z(t)$,  equation \eqref{aux1} becomes
\[
\lambda e^{\lambda t}  z(t) + e^{\lambda t} \dot z(t) + f(e^{\lambda t}z(t))= g(t), \quad z(0) = x_0,
\]
hence
\begin{equation} \label{aux2}
\dot z(t) + e^{-\lambda t} f(e^{\lambda t}z(t)) + \lambda z(t)= e^{-\lambda t}g(t).
\end{equation}
Setting
\[
Q(t,z)(t) = e^{-\lambda t} f(e^{\lambda t}z(t)) + \lambda z(t),
\]
we have for  $z$ and $v$ in $\mathbb{R}^n$, and $t\ge0$ by \eqref{lowerbound}

\[
(DQ( t,z)v,v) \ge (Df(z)v +\lambda v,v) \ge 0.
\]
Thus $Q(t,\cdot)$ is monotone for each $t\ge0$, i.e. we have
\[
\langle Q(t,z)-Q(t,y), z-y \rangle \ge 0.
\]
The differential equation for $z$ now reads
\begin{equation} \label{aux3}
\dot z(t) + Q(t,z(t))= h(t), \quad z(0) = x_0
\end{equation}
with $h(t) = e^{-\lambda t} g(t)$.

(ii) {\it A priori estimate.} Let $z \in H^1(0,T;\mathbb{R}^n)$ be a solution of \eqref{aux3}. After multiplication by $z$ and integration,
\[
\begin{array}{l}
\displaystyle \int_0^t \dot z\cdot z \,ds +
\displaystyle \int_0^t (Q(s,z(s)) - Q(s,0))\cdot (z(s) - 0)\, ds = \\
\displaystyle \hspace{1cm}- \int_0^t Q(s,0)\cdot z(s)\,ds +  {\int_0^t h(s)\cdot z(s)\, ds}
\end{array}
\]
By the monotonicity of $Q$ and Young's inequality
{\begin{equation} \label{aprioriest}
\frac{1}{2} |z(t)|^2 \le \frac{1}{2} \left( |x_0|^2 + \int_0^t(|Q(s,0)|^2+|h(s)|^2)\, ds+   \int_0^t
| z(s)|^2\,ds \right) .
\end{equation}
}
Gronwall's inequality implies  that
\begin{equation}\label{eq:kk1}
 |z(t)| \le (|x_0| +\|Q(\cdot,0)\|_{L^2(0,T;\mathbb{R})} +\|h\|_{L^2(0,T;\mathbb{R})})e^{\frac{1}{2}T}=: R \quad \mbox{ a.e. on } [0,T],
\end{equation}
and consequently
\begin{equation}\label{eq:kk2}
 |x(t)| \le (|x_0| + \sqrt{T} |f(0)|  +\|g\|_{L^2(0,T;\mathbb{R})})e^{(\frac{1}{2}+ \lambda)T}\quad \mbox{ a.e. on } [0,T].
\end{equation}
Since $f$ is continuously differentiable, $f$ is locally Lipschitz, i.e. Lipschitz on compact sets of $\mathbb{R}^n$. Moreover, we have the a priori estimate above. Thus  existence and uniqueness of a solution of \eqref{aux1} can be obtained by the principle ''extension or blow up''. We refer e.g. to Corollary 3.9 of \cite{Zeidler1986}.
\end{proof}
Now we are able to deal with the delay differential equation.
We refer to $x$ as a solution to  \eqref{delayeq}, if $x\in C([-b,T],\mathbb{R}^n)$ with $x|_{[0,T]} \in H^1([0,T];\mathbb{R}^n)$, $x|_{[-b,0]} = \varphi$, and  \eqref{delayeq} is satisfied a.e. in $(0,T)$. Unless necessitated for reasons of clarity we shall henceforth not distinguish between $x$ as solution on $[0,T]$ or on $[-b,T]$.
\begin{theorem}[Existence and uniqueness] \label{well-posedness} If $f$, $\varphi$, and $g$ satisfy 
Assumption \ref{A1}, then the delay equation \eqref{delayeq} has a unique solution 
$x \in \raumHeins$. If moreover $g\in H^1(0,T;\mathbb{R}^n)$, then $x \in  H^2(0,T;\mathbb{R}^n)$.
\end{theorem}
\begin{proof} With Proposition \ref{L1} at hand the verification of this result can be obtained in a standard  manner proceeding stepwise in time with stepsize $\tau$.  
\end{proof}
\begin{remark} For the second-order differentiability of the solution $x$ with respect to the delay $\tau$, 
depending on the function space setting to be chosen, the higher regularity $x \in H^2(-b,T;\mathbb{R}^n)$ is required. Even for $\varphi \in H^2(-b,T;\mathbb{R}^n)$ and $g \in H^1(0,T;\mathbb{R}^n)$, this needs a compatibility condition at $t = 0$: 

Indeed, if $x \in H^2(-b,T;\mathbb{R}^n)$, then $\dot x$ has to be continuous at $t=0$. We have 
\[
\dot x(0^-) = \lim_{t \uparrow 0} \dot x(t) =  \lim_{t \uparrow 0} \dot \varphi(t) = \dot \varphi(0)
\]
and
\[
\begin{aligned}
\dot x(0^+)  = \lim_{t \downarrow 0} \dot x(t) &=  
\lim_{t \downarrow 0} (-f(x(t)) + A x(t-\tau) + g(t)) \\
&= - f(\varphi(0)) + A \varphi(-\tau) + g(0).
\end{aligned}
\]
Therefore, to have $x \in H^2(-b,T;\mathbb{R}^n)$, the compatibility condition 
\begin{equation} \label{compat}
\dot \varphi(0) = - f(\varphi(0))  + A \varphi(-\tau) + g(0)
\end{equation}
is needed.
\end{remark}

\begin{remark} \label{Rsemigroup}
Let us point out that the compatibility condition also naturally arises if the delay equation \eqref{delayeq} is treated as abstract equation in function space over the interval $(-b,0)$. To briefly explain the context let us consider the linear, homogenous case, with $f(x)=A_0 \, x$ and $g=0$. For the function space setting, there are two natural choices, namely $C(-b,0;\mathbb{R}^n)$ or $\mathbb{R}^n \times L^2(-b,0;\mathbb{R}^n)$. Choosing the former, we define the infinitesimal generator $\mathcal{A}$ associated to   \eqref{delayeq} by   $\mathcal A y = \frac{d}{ds} y$ with $\text{dom} (\mathcal{A})= \{y \in  C^1(-b,0;\mathbb{R}^n): \frac{d}{ds}y(0) = A_0 y(0) + Ay(-\tau)\}$,   see e.g. \cite[Section 2 and Section19]{hale1971}.

%

  The abstract equation associated to \eqref{delayeq} is then given by
\begin{equation*}
\frac{d}{dt} x(t) = \mathcal{A} x(t), \quad \text{ with } x(0) = \varphi.
\end{equation*}
The semigroup $ e^{\mathcal{A}t}$ generated by $\mathcal{A}$ satisfies $ e^{\mathcal{A}t}\varphi = x(t+\cdot)$  on $(-b,0)$, for all $t\ge 0$, with $x$ the solution that we discussed above. Moreover $e^{\mathcal{A}t} \varphi \in \text{dom} (\mathcal{A})$ for all $t\ge b$. Thus the compatibility condition is satisfied for all $t\ge b$.
\end{remark}

\section{Differentiability with respect to the time delay $\tau$} \label{S3} 
~
By Theorem \ref{well-posedness}, for each $\tau \in [0,b]$ the delay equation \eqref{delayeq} has a unique solution $x$ that we denote by $x[\tau]$. The mapping $\tau \mapsto  x[\tau]$ is well defined from $[0,b]$ to $C([-b,T],\mathbb{R}^n)$
and to $H^1(-b,T;\mathbb{R}^n)$, if $\varphi \in H^1(-b,0;\mathbb{R}^n)$. In the remainder of this section,  we discuss the first derivative of the mapping $\tau \mapsto x[\tau]$.

In principle, we might adapt the proof of an analogous theorem of differentiability from Casas et al.  \cite{casas_mateos_troeltzsch2019} that was performed for the optimization of time delays in semilinear parabolic equations with time delay. Here, we present a different proof via the implicit function theorem. We can benefit from this strategy also for the second derivative. 

To this end, following Hale and Ladeira \cite{hale_ladeira1993}, we transform equation \eqref{delayeq} in the following way: We set
\[
\phi(t) = \left\{
\begin{array}{ll}
\varphi(t), &t \in [-b,0],\\
\varphi(0), & t \in (0,T],
\end{array}
\right.
\]
and
\[
z(t) = x(t) - \phi(t), \quad t \in [-b,T].
\]
We observe that $\phi \in  H^1(-b,T;\mathbb{R}^n)$,  $z(t)=0$ on $[-b,0]$, and $x(t) = z(t) + \phi(t)$. 
For convenience, we introduce the following subspace of $H^1(-b,T;\mathbb{R}^n)$:
\[
H_{[0]}^1(-b,T;\mathbb{R}^n) = \{z \in H^1(-b,T;\mathbb{R}^n): z(t) = 0 \mbox{ in } [-b,0]\}.
\]
In addition, we define $F: H_{[0]}^1(-b,T;\mathbb{R}^n) \times [0,b] \to H_{[0]}^1(-b,T;\mathbb{R}^n)$ by

\begin{equation}\label{def F}
(F(z,\tau))(t) = \left\{
\begin{array}{ll}
0, & t\in [-b,0],\\[1ex]
\displaystyle \int_0^t \big\{(-f(z + \phi) + g)(s) + (A(z + \phi))(s-\tau))\big\}\, ds, & t \in [0,T].
\end{array}
\right.
\end{equation}
Then \eqref{delayeq} for $x$ is equivalent to the equation for $z \in H_{[0]}^1(-b,T;\mathbb{R}^n)$,
\begin{equation}\label{delayeqz}
z(t) = (F(z,\tau))(t), \ t \in [-b,T].
\end{equation}
This transformation justifies to work in the  closed subspace $\raumH$ of $H^1(-b,T;\mathbb{R})$.

By Theorem \ref{well-posedness} and the equivalence of \eqref{delayeqz} with \eqref{delayeq}, the mapping $[0,b] \ni \tau \mapsto z \in H_{[0]}^1(-b,T;\mathbb{R}^n)$ is well defined.
To express the dependency of this solution on $\tau$, we denote it by $z[\tau]$.  To study its differentiability properties, we use the following notation:
\[
\begin{aligned}
\dot z[\tau](t) := \partial_t z[\tau](t), & \qquad \ddot z[\tau](t) := \partial_t^2 z[\tau](t) \\
z'[\tau](t) : =  \partial_\tau z[\tau](t), & \qquad z''[\tau](t) : =  \partial_{\tau}^2 z[\tau](t).
\end{aligned}
\]
\begin{lemma} \label{Lshift} The parameterized shift mapping $S: (z,\tau) \mapsto z(\cdot-\tau)$ is continuously Fr\'echet-differentiable from 
$H_{[0]}^1(-b,T;\mathbb{R}^n) \times [0,b\fredi{]}$ to $L^2(0,T;\mathbb{R}^n)$. The derivative is
\begin{equation} \label{F-derivative}
(DS(z,\tau)(h,\delta))(t) =  h(t-\tau)-\dot z(t-\tau)\delta, \quad t \in [0,T].
\end{equation}
\end{lemma}
\begin{proof}
We first confirm that  \eqref{F-derivative} is the Fr\'{e}chet derivative of $(z,\tau) \mapsto z(\cdot-\tau)$:
Let $0 \le \tau < b$ and $|\delta| < b -\tau$ so that $\tau + \delta \le b$. We have
\[
\begin{aligned}
&(S(z+h,\tau + \delta)- S(z,\tau))(t)= (z+h)(t-(\tau + \delta)) - z(t-\tau) \\
&\qquad = z(t-\tau - \delta) - z(t-\tau) +   h(t-\tau) + h(t-\tau - \delta)-h(t-\tau)\\ 
&\qquad = h(t-\tau) - \int_0^1 \dot z(t-\tau - s \delta) \,\delta  ds + \int_0^1\dot h(t-\tau - s \delta) \delta\, ds
 .\\
&\qquad = h(t-\tau) - \dot z(t-\tau) \delta   -   \delta \int_0^1 (\dot z(t-\tau - s \delta) - \dot z(t-\tau)) \, ds + R_h(h,\delta)\\
&\qquad = h(t-\tau) - \dot z(t-\tau) \delta + R_z(h,\delta) + R_h(h,\delta), \\
\end{aligned}
\]
where the remainder terms $R_z$ and $R_h$ are defined by 
\[
\begin{aligned}
&R_h(h,\delta) =  \delta  \, \int_0^1\dot h(t-\tau - s \delta)\, ds,\\
&R_z(h,\delta)= \delta \int_0^1 (\dot z(t-\tau - s \delta) - \dot z(t-\tau)) \, ds.
\end{aligned}
\]
Here, we have used that
$
\partial_s z(t-s) = - \dot z(t-s)
$
which follows from the definition of the weak derivative $\dot z(t-\tau)$ via testing with a smooth function.

For convenience, in this proof we introduce the abbreviations
\[
\|\cdot\|_{H_{[0]}^1}:=\|\cdot\|_{H_{[0]}^1(-b,T;\mathbb{R}^n)}, \qquad \|\cdot\|_{L^2} := \|\cdot\|_{L^2(0,T;\mathbb{R}^n)}.
\]
The $L^2$-norm of the remainder terms $R_z$ and $R_h$, divided by $\|h\|_{H_{[0]}^1}+|\delta|$, tends to zero, if $\delta \to 0$:
\[
\begin{aligned}
\|R_h\|_{L^2}^2 &= \int_{0}^T\left|\int_0^1 \dot h(t-\tau-s\delta)\,ds\right|^2\delta^2dt\\
&\le  \int_{0}^T\int_0^1 |\dot h(t-\tau-s\delta)|^2\,ds\,dt\, \delta^2 = \int_0^1\int_{0}^{T-\tau-s\delta}|\dot h(\sigma)|^2d\sigma \, ds \, \delta^2\\
&\le \int_0^1\int_{0}^T\fredi{|}\dot h(\sigma)\fredi{|}^2 d\sigma \, \delta^2 = \|h\|_{H_{[0]}^1}^2 \delta^2,
\end{aligned}
\]
notice that $h(\sigma)=0$ for $\sigma \le 0$.
Therefore $\|R_h\|_{L^2} \le \delta
\|h\|_{H_{[0]}^1} $
and hence 
\begin{equation}\label{eq:aux1}
\frac{\|R_h\|_{L^2}}{\|h\|_{H_{[0]}^1}+|\delta|} \to 0 \mbox{ if } \|h\|_{H_{[0]}^1}+|\delta| \to 0.
\end{equation}
Analogously, we obtain
\[
\frac{1}{\delta^2}\|R_z(h,\delta)\|_{L^2}^2 \le \int_0^1 \int^T_{0} |\dot z(t-\tau - s \delta) - \dot z(t-\tau)|^2\, dtds.
\]
The function $\dot z$ belongs to $L^2(-b,T;\mathbb{R}^n)$, and hence by the continuity of the shift operator in $L^2(-b,T;\mathbb{R}^n)$, see \cite[pg. 199]{hewitt_stromberg1965}  we obtain
\[
\frac{\|R_z(h,\delta)\|_{L^2}}{\|h\|_{H_{[0]}^1}+|\delta|} \le \frac{1}{|\delta|}\|R_z(h,\delta)\|_{L^2} \to 0, \mbox{ if }   \|h\|_{H_{[0]}^1}+|\delta| \to 0.
\]
The properties of the remainder terms confirm that \eqref{F-derivative} is the expression of the Fr\'echet derivative of the shift mapping $S$. 
The derivative depends continuously on $(z,\tau)$: Indeed, we have     
\[
\|(DS(z,\tau)-D(S(y,\sigma))(h,\delta)\|_{L^2} \le \|(\dot z(\cdot-\tau) - \dot y(\cdot-\sigma))\|_{L^2}|\delta| + \|h(\cdot-\tau)-h(\cdot-\sigma)\|_{L^2}.\\
\]
The second term tends to 0 as $|\tau-\sigma|\to 0$ with the same argument as the one which led to \eqref{eq:aux1}.  For the first one we estimate
\[
 \|(\dot z(\cdot-\tau) - \dot y(\cdot-\sigma))\|_{L^2}\le  \|(\dot z(\cdot-\tau) - \dot y(\cdot-\tau))\|_{L^2}+ \|(\dot y(\cdot-\tau) - \dot y(\cdot-\sigma))\|_{L^2}.
\]
For $y \to z$ in $\raumH$, the first term obviously tends to zero. For $\sigma \to \delta$, the second term tends to zero by the continuity of the shift operator in $L^2$.
These estimates show the continuity of the derivative. In the case  $\tau = b$, we assume $\delta < 0$ and obtain the result for  the left derivative of $S$ in $b$.
\end{proof}

{\bf Notation.}  Preparing the next results, we introduce the following mappings defined in $\mathcal{H}= H^1_{[0]}(-b,T;\mathbb{R}^n)\times [0,b)$, namely $G: \mathcal{H} \to L^2(0,T;\mathbb{R}^n)$ and $\mathcal{F}: \mathcal{H} \to H^1(-b,T;\mathbb{R}^n)$ defined by
\[
\begin{array}{l}
G(z,\tau) = \big\{-f(z+\phi) + A(z(\cdot-\tau)+\phi(\cdot-\tau))\big\}{|_{[0,T]}} + g,\\[2ex]
\mathcal{F}(z,\tau) = z - F(z,\tau),
\end{array}
\]
where $F$ is defined in \eqref{def F}.
Notice that
\[
F(z,\tau)(t) = \int_0^t G(z,\tau)(s) \, ds, \quad \forall t \in [0,T].
\] 
The space $H_{[0]}^1(-b,T;\mathbb{R}^n)$ is continuously embedded in $C([-b,T],\mathbb{R}^n)$ and the superposition operator $v \mapsto f(v)$ is of class $C^1$ in $C([-b,T],\mathbb{R}^n)$, because $f: \mathbb{R}^n \to \mathbb{R}^n$ is of class $C^1$. Moreover, $\phi$ belongs to  $H^1(-b,T;\mathbb{R}^n)$. Therefore,  the mapping  $z \mapsto f(z+\phi)$ is of class $C^1$ from $H_{[0]}^1(-b,T;\mathbb{R}^n)$ to $C([-b,T],\mathbb{R}^n) \hookrightarrow L^2(-b,T;\mathbb{R}^n)$. 

Thanks to Lemma \ref{Lshift} and the differentiability of $f$, the operator $G$ is of class $C^1$ from $\mathcal{H}$
to $L^2(0,T;\mathbb{R}^n)$. Therefore, via integration, $F$ is class $C^1$ from $\mathcal{H}$ to $H^1_{[0]}(-b,T;\mathbb{R}^n).$

In view of these arguments, we have proved the following result:
\begin{lemma}
The mapping $(z,\tau) \mapsto F(z,\tau)$ is continuously differentiable from $H_{[0]}^1(-b,T;\mathbb{R}^n) \times [0,b\fredi{]}$ to $H_{[0]}^1(-b,T;\mathbb{R}^n)$.
\end{lemma}
\begin{theorem} \label{differentiability} The mapping $\tau \mapsto z[\tau]$ is continuously differentiable from $[0,b\fredi{]}$ to $H^1_{[0]}(-b,T;\mathbb{R}^n)$.
\end{theorem}
\begin{proof} The function $z[\tau]$ is the unique solution of the equation $\mathcal{F}(z,\tau)=0$. Notice that existence and uniqueness of $z[\tau]$ follow from Thm. \ref{well-posedness}.  With $F$, also $\mathcal{F} = I -F$ is of class $C^1$ in $\raumH$. To show the result, we invoke the implicit function theorem. 

Therefore, we confirm that $D_z \mathcal{F}(z,\tau)$ is an isomorphism. We have $D_z \mathcal{F}(z,\tau) = I - D_z F(z,\tau)$, hence we have to consider the equation 
\[
v - D_z F(z,\tau) v = d
\]
in $H^1_{[0]}(-b,T;\mathbb{R}^n)$. More detailed, this equation for $v \in H^1_{[0]}(-b,T;\mathbb{R}^n)$ reads
\[
v(t) + \int_0^t\left\{ Df(z(s)+\phi(s))v(s) - A v(s-\tau)\right\}ds = d(t), \qquad t \in [0,T],
\]
or equivalently
\[
\begin{aligned}
&\dot v(t) +  Df(z(t)+\phi(t))v(t)= A v(t-\tau) + \dot d(t),\quad t \in (0,T],\\[1ex]
&v(t) = 0, \quad t \in [-b,0].
\end{aligned}
\]
For each $d \in \raumH$, this linear delay equation has a unique solution $v \in \raumH$. This can be shown stepwise in time, analogously to Theorem \ref{well-posedness}. The arguments are  even simpler, because we can use a standard existence and uniqueness theorem for systems of linear ordinary differential equations. The mapping $\dot d \mapsto v$ is continuous from $L^2(0,T;\mathbb{R}^n)$ to $\raumH$ and hence $D_z \mathcal{F}$ is an isomorphism. 

Since $\mathcal{F}$ is of class $C^1$, the desired result follows from the implicit function theorem.
\end{proof}

\begin{corollary} \label{Cor_x'}
The mapping $\tau \mapsto x[\tau]$ is continuously differentiable from $[0,b\fredi{]}$ to 
$H^1(-b,T;\mathbb{R}^n)$. Its derivative $w[\tau]:=x'[\tau]$ is the unique solution of the delay equation 
\begin{equation} \label{equation_for_x'}
\begin{array}{ll}
\partial_t w(t)+Df(x[\tau](t))w(t) =  A w(t-\tau) - A \dot x[\tau](t-\tau),& t \in (0,T],\\[1ex]
w(t) = 0, & t \in [-b,0].
\end{array}
\end{equation}
Moreover we have
\begin{equation}  \label{eq:k11}
\partial_t\partial_\tau x[\tau](\cdot)= \partial_{\tau}\partial_t x[\tau](\cdot) \text{ in } L^2(0,T;\mathbb{R}^n).
\end{equation}
\end{corollary}
\begin{proof} Thanks to our transformation, we have $x[\tau] = z[\tau]+\phi$.
Therefore, the differentiability properties of $\tau \mapsto z[\tau]$ transfer to $\tau \mapsto x[\tau]$ and we have
$x'[\tau]=z'[\tau]$. The equation for $x'[\tau]$ can be determined by implicit differentiation; $x[\tau]$ obeys
\[
\begin{array}{lcl}
\displaystyle x[\tau](t) = \varphi(0) + \int_0^t \big\{-f(x[\tau](s))+ A x[\tau](s-\tau) + g(s)\big\}\, ds,& t \in (0,T],\\[1ex]
x[\tau](t) =0, & t \in [-b,0].
\end{array}
\]
Theorem \ref{differentiability} justifies to differentiate both equations with respect to $\tau$, hence
\[
\begin{aligned}
& x'[\tau](t)=\displaystyle \int_0^t \big\{-(Df(x[\tau])x'[\tau])(s)
+ A x'[\tau](s-\tau) - A \dot x[\tau](s-\tau)\big\}\,ds, \, t \in (0,T],\\[1ex]
&x'[\tau](t) =0 ,  \, t \in [-b,0].
\end{aligned}
\]
In view of Theorem \ref{differentiability}, the function  $x'[\tau]$ belongs to $\raumHeins$. We can differentiate
the first equation w.r. to $t$ and obtain the claimed result of the corollary.

To verify \eqref{eq:k11}, note that 
\begin{equation}\label{eq:k12}
\partial_t x[\tau](t) = -f(x[\tau](t)) - Ax[\tau](t-\tau) +g(t).
\end{equation}
The right hand side is differentiable with respect to $\tau$ and belongs to $L^2(0,T;\mathbb{R}^n)$. Hence $\partial_\tau x[\tau](\cdot)$ exists as an element in  $L^2(0,T;\mathbb{R}^n)$. Finally \eqref{eq:k11} follows by taking the derivative with respect to $\tau$ in \eqref{eq:k12} and comparing with \eqref{equation_for_x'}.
\end{proof}

\section{Optimization of the time delay} \label{S4}

In this section, we apply the theory of the previous sections to the optimization problem 
\begin{equation} \label{OP}
\min_{\tau \in [0,b]} j(\tau) : = \frac{1}{2} \int_0^T \big|x[\tau](t) - x_d(t)\big|^2 \, dt,
\end{equation}
where $x_d \in L^2(0,T;\mathbb{R}^n)$ is a given desired state and $x[\tau]$ denotes the solution of 
\eqref{delayeq} for given $\tau$. 

We discuss the first- and second-order sensitivity of the cost function $j$ and derive first- and second-order optimality conditions. 
The second-order sensitivity analysis of  $j$ is performed in two ways. In the first, we
use the second-order derivative $x''[\tau]$, in the second we invoke an adjoint calculus that does not exploit the derivative $x''[\tau]$.

\subsection{First-order sensitivity analysis} \label{S51}

We first assume $\varphi \in H^1(-b,0;\mathbb{R}^n)$; then equation \eqref{delayeq} admits a unique
solution $x[\tau] \in \raumHeins$. If $g$ belongs to $H^1(0,T;\mathbb{R}^n)$, then we have 
$x \in H^2(0,T;\mathbb{R}^n)$.

Associated to $x[\tau]$, we define the adjoint equation
\begin{equation} \label{E1}
\left\{
\begin{aligned}
&-\dot p(t) + Df(x[\tau](t))^\top p(t)= A^\top p(t +\tau) + x[\tau](t) - x_d(t), \quad t \in [0,T),\\[1ex]
&\quad p(t) = 0, \quad t \in [T,T+b].
\end{aligned}
\right.
\end{equation}
This equation admits a unique solution $p \in H^1(0,T+b;\mathbb{R}^n)$, denoted by $p[\tau]$.
For the sake of brevity, we sometimes omit the dependence on $\tau$. Concerning the differentiability
of $p[\tau]$ with respect to $\tau$, we have the following result analogously to Corollary \ref{Cor_x'}:
\begin{proposition} \label{P1} If $f\in C^2(\mathbb{R}^n,\mathbb{R}^n)$, the mapping $\tau \mapsto p[\tau]$ is continuously  differentiable from $[0,b)$
to $H^1(0,T+b;\mathbb{R}^n)$. Its derivative $w=p'[\tau]$ is the unique solution of 
\begin{equation} \label{E2}
\left\{
\begin{aligned}
&-\dot w(t) + Df(x[\tau](t))^\top w(t) - A^\top w(t +\tau) =  \\[1ex]
&  \qquad -x'[\tau] D(Df(x[\tau](t))^\top) p[\tau](t)    + A^\top \dot p[\tau](t+\tau) + x'[\tau](t), \quad t \in [0,T],\\[1ex]
&\quad p(t) = 0, \quad t \in [T,T+b],
\end{aligned}
\right.
\end{equation}
where 
\[
(q D(Df(x[\tau](t))^\top) p)_i = \sum_{j,k =1}^n (f_j)_{x_i x_k} p_j q_k.
\]
\end{proposition}
The proof is similar to that of Corollary \ref{Cor_x'} with two differences: Now, we have a backward equation. This can be reduced to a forward equation by a standard transformation of time. Moreover, 
in  Corollary \ref{Cor_x'} the right-hand side $g$ did not depend on $\tau$. Here, the right-hand side is $x[\tau]-x_d$.

The first derivative of the cost $j$ is characterized next. Here and in what follows, $\langle \cdot , \cdot \rangle$ denotes the standard inner product of $\mathbb{R}^n$.

\begin{proposition} \label{P2}
If $f$ is continuously differentiable, $\varphi \in H^1(-b,0;\mathbb{R}^n)$, and $g \in L^2(0,T;\mathbb{R}^n)$, then $j \in C^1[0,b]$ and 
\begin{equation} \label{Dj*}
j'(\tau) = - \int_0^T \langle p[\tau], A \dot x[\tau](t-\tau)\rangle dt.
\end{equation}
\end{proposition}
\begin{proof} We compute, not indicating the dependence of $p[\tau]$ on $\tau$,
\begin{align*}
& j'(\tau) = \int_0^T \langle x[\tau](t)-x_d(t), x'[\tau](t)\rangle dt\\
&= \int_0^T  \langle -\dot p(t) + Df(x[\tau](t))^\top p(t)- A^\top p(t +\tau) , x'[\tau](t)\rangle dt\\
&= \int_0^T  \langle p(t), \dot x'[\tau](t) + Df(x[\tau](t)) x'[\tau](t) - A x'[\tau](t -\tau) \rangle dt\\
&= - \int_0^T  \langle p(t), A \dot x[\tau](t -\tau) \rangle dt.
\end{align*}
\end{proof}
\subsection{Second-order sensitivity analysis for $j$} \label{S52}
In this section we verify that under additional assumptions on the problem data $f, \varphi$ and $g$, the cost functional is twice continuously differentiable. This allows us to formulate a second-order sufficient optimality condition for \eqref{OP}. 
We will rely on the following 
\begin{assumption} \label{A3} The function $f$ belongs to $C^2(\mathbb{R}^n,\mathbb{R}^n)$, $\varphi$ to 
$H^2(-b,0;\mathbb{R}^n)$, and $g$ to $H^1(0,T;\mathbb{R}^n)$.
\end{assumption}
\begin{proposition} \label{P3}
If Assumption \ref{A3} holds,
then $j \in C^2[0,b]$ and 
\begin{equation}\label{D2j}
\begin{aligned}
&\displaystyle j''(\tau) =\displaystyle \int_0^T \big|x'[\tau]\big|^2\,dt -  \int_0^T \langle p[\tau](t) ,D^2f(x[\tau])(x'[\tau],x'[\tau])\rangle dt\\[1ex]
&\qquad \displaystyle -2 \int_\tau^T \langle p[\tau](t), A \dot x'[\tau](t-\tau)\rangle dt + \langle p[\tau](\tau), A\big(\dot x[\tau](0^+)-\dot \varphi(0)\big)\rangle\\[1ex]
&\qquad \displaystyle +\int_0^\tau \langle p[\tau](t), A \ddot \varphi(t-\tau)\rangle dt + \int_\tau^T\langle p[\tau](t),A \ddot x[\tau](t-\tau)\rangle dt.
\end{aligned}
\end{equation}
\end{proposition}

\begin{proof}
For the second derivative, we obtain
\begin{align*}
& j''(\tau) = \frac{d}{d\tau}\left[  
-\int_0^\tau \langle p[\tau](t), A \dot x[\tau](t -\tau) \rangle dt -\int_\tau^T \langle p[\tau](t), A \dot x[\tau](t -\tau) \rangle dt\right]\\
&=-\int_0^T \langle p'[\tau](t), A \dot x[\tau](t -\tau) \rangle dt - \int_0^T \langle p[\tau](t), A \dot x'[\tau](t -\tau) \rangle dt\\
&+ \langle p[\tau](\tau),A (\dot x[\tau](0^+) - \dot \varphi(0))\rangle + \int_0^\tau \langle p[\tau](t), A \ddot \varphi(t -\tau) \rangle dt \\
&+ \int_\tau^T \langle p[\tau](t), A \ddot x[\tau](t -\tau) \rangle dt \\
\end{align*}
\begin{align*}
&=-\int_0^T \langle p'[\tau](t), A \dot x[\tau](t -\tau) \rangle dt + \int_0^T \langle p[\tau](t), A \dot x'[\tau](t -\tau) \rangle dt\\
&-2  \int_0^T \langle p[\tau](t), A \dot x'[\tau](t -\tau) \rangle dt + \langle p[\tau](\tau),A (\dot x[\tau](0^+) - \dot \varphi(0))\rangle\\
& + \int_0^\tau \langle p[\tau](t), A \ddot \varphi(t -\tau) \rangle dt + \int_\tau^T \langle p[\tau](t), A \ddot x[\tau](t -\tau) \rangle dt.
\end{align*}
Let us turn to the first two terms on the right-hand side of the last expression:
\begin{align*}
&-\int_0^T \langle p'[\tau](t), A \dot x[\tau](t -\tau) \rangle dt + \int_0^T \langle p[\tau](t), A \dot x'[\tau](t -\tau) \rangle dt\\
&=\int_0^T \langle p'[\tau](t), \dot x'[\tau](t) + Df(x[\tau](t)) x'[\tau](t) - A x'[\tau](t -\tau) 
\rangle dt\\
&\hspace{2cm}+\int_0^T \langle p[\tau](t), A \dot x'[\tau](t -\tau) \rangle dt \\
&= \int_0^T \langle -\dot p'[\tau](t)  + Df(x[\tau](t))^\top p'[\tau](t) - A^\top p'[\tau](t +\tau) ,x'[\tau](t)\rangle dt\\
&\hspace{2cm}+\int_0^T \langle p[\tau](t+\tau), A \dot x'[\tau](t) \rangle dt \\
&=\int_0^T\langle \dot p[\tau](t+\tau), A x'[\tau](t)\rangle dt + \int_0^T \big|x'[\tau](t)\big|^2 dt +\int_0^T\langle p[\tau](t+\tau), A \dot x'[\tau](t)\rangle dt \\
&- \int_0^T \langle x'[\tau](t) D(D f(x[\tau](t))^\top p[\tau](t) , x'[\tau](t)\rangle dt\\
&= \int_0^T \big|x'[\tau](t)\big|^2 dt - \int_0^T \langle x'[\tau](t) D(D f(x[\tau](t))^\top p[\tau](t) , x'[\tau](t)\rangle dt,
\end{align*}
where we used that 
the action of the tensor $D^2f(x)$ is given by
$$D^2f(x)(h_1,h_2)= col_k \sum_{i,j=1}^n h_1^\top D^2f_k(x) \, h_2 , \text{ for } h_1\in \mathbb{R}^n, h_2\in \mathbb{R}^n,$$
and
\[
\langle D^2f(x)(v, v) ,p \rangle = \langle v D(D f(x)^\top) p, v \rangle \quad \forall \,v,\, p \in \mathbb{R}^n.
\]
\end{proof}
\begin{corollary}If the compatibility condition \eqref{compat} is satisfied, then
\begin{equation}
\begin{aligned}
& j''(\tau) =  \int_0^T \big|x'[\tau](t)\big|^2 dt - \int_0^T \langle p[\tau](t),(D^2f(x[\tau])(x'[\tau],x'[\tau]))(t)\rangle dt\\
&\qquad \quad \ -2 \int_0^T\langle p[\tau](t), A \dot x'[\tau](t-\tau)\rangle dt + \int_0^T\langle p[\tau](t), A \ddot x[\tau](t-\tau)\rangle dt.
\end{aligned}
\end{equation}
\end{corollary}

\subsection{Existence for \eqref{OP} and first/second-order optimality } \label{S4.5}
With the results of the previous sections, the analysis of \eqref{OP} is now completely standard. We summarize it in the following theorem.

\begin{theorem}\label {theo4.1} With Assumption \ref{A1} holding there exists a solution $\bar \tau$ of \eqref{OP}, satisfying the first-order condition $j'(\bar \tau)(\tau - \bar \tau) \ge 0$ for all $\tau \in [0,b]$. If moreover the regularity assumptions of Proposition \ref{P3} hold, then  each of the following conditions is sufficient for $\hat \tau$ to be a strict local minimizer of $j$: 

(i) $0 < \hat \tau < b$, $j'(\hat \tau)=0$, and $j''(\hat \tau)>0$,

(ii) $\hat \tau=0$ and $j'(0)>0$ or $\hat \tau = b$ and $j'(b)>0$.
\end{theorem}

\section{Second-order derivative of the state} \label{S5}

In this section, we discuss the second-order derivative of the mapping $\tau \mapsto x[\tau]$ for the equation \eqref{delayeq}. We prove the existence of $x''[\tau]$ in $L^2(0,T;\mathbb{R}^n)$ and establish equations for it.  This  allows us to obtain an alternative expression for the second derivative of the cost:
\begin{equation} \label{j''}
\begin{aligned}
j''(\tau) &= \frac{d}{d\tau} j'[\tau] =  \frac{d}{d\tau}\int_0^T \langle x[\tau](t)-x_d(t), x'[\tau](t)\rangle dt\\[1ex]
&=\int_0^T \big|x'[\tau](t)\big|^2dt + \int_0^T \langle x[\tau](t)-x_d(t),x''[\tau](t)\rangle dt.
\end{aligned}
\end{equation}
This requires some attention since  $t \mapsto x''[\tau](t)$ is not differentiable
at $t=\tau$ unless the compatibility condition \eqref{compat} is satisfied. The following example illustrates the difficulty:

\begin{example} \label{E5.1} Consider for $n=1$ and $0 < \tau < 1$ the linear delay equation
\begin{equation} \label{delay_simple}
\begin{aligned}
&\dot x(t) = x(t-\tau), &\ t > 0,\\
&x(t) = 1, & -1\le t \le 0.
\end{aligned}
\end{equation}
Here, we have $\varphi(t) = 1, \quad t\in [-1,0].$ Solving this equation stepwise on $[0,\tau]$, $[\tau,2\tau]$, and 
$[2\tau,3\tau]$, we find
\[
x[\tau](t) = \left\{
\begin{array}{ll}
1,&t \in [-1,0],\\[1ex]
t+1,&t \in (0,\tau],\\[1ex]
\frac{1}{2} (t-\tau)^2 + t + 1,&t \in (\tau,2\tau],\\[1ex]
\frac{1}{6}(t-2\tau)^3 +\frac{1}{2}(t-\tau)^2 + t+1,& t \in (2\tau,3\tau].
\end{array} 
\right.
\]
Differentiating $x[\tau]$ w.r. to $\tau$ in the single subintervals, we get
\[
x'[\tau](t) = \left\{
\begin{array}{ll}
0,&t \in [-1,\tau],\\[1ex]
-(t-\tau),&t \in (\tau,2\tau],\\[1ex]
- (t-2\tau)^2 -(t-\tau),& t \in (2\tau,3\tau].
\end{array} 
\right.
\]
This is a function of $H^1(-1,3)$. Differentiating again, we arrive at
\[
x''[\tau](t) = \left\{
\begin{array}{ll}
0,&t \in [-1,\tau],\\[1ex]
1,&t \in (\tau,2\tau],\\[1ex]
4 (t-2\tau) + 1,& t \in (2\tau,3\tau].
\end{array} 
\right.
\]
We see that $x''[\tau]$ exhibits a jump at $t = \tau$. It exists as a well defined function of $L^2(-1,3\tau)$, but we cannot differentiate it with respect to $t$ in $t = \tau$. Therefore, we cannot have a standard differential equation to determine  $x''[\tau]$ on the whole interval $[0,T]$.
In our example, the function $x[\tau]$ belongs to $H^1(-1,3)$, but its derivative $\dot x[\tau]$ is discontinuous at $t=0$. Note that  $\varphi$ does not satisfy the compatibility condition \eqref{compat}.
\end{example}

\begin{remark}We have differentiated $x[\tau]$ and $x'[\tau]$ on single subintervals of time. It is not obvious that this stepwise differentiation leads to a correct result, because the values $\tau$ and $2\tau$ need special care. Here, the computation is correct, because $x[\tau]$ and $x'[\tau]$ belong to $H^1(-1,3)$, see also Lemma \ref{piecewise*}.
\end{remark}

In view of the example, we will study the second-order differentiability of $x[\tau]$ first by differentiating the integral version of equation \eqref{equation_for_x'} with respect to $\tau$,
\begin{equation}\label{integral_equation_for_x'}
w(t) = \int_0^t \{-Df(x[\tau](s))w(s) + A w(s-\tau) - A\dot x[\tau](s-\tau)\}ds,
\end{equation}
where 
\[
w\in \Lnull = \{ w \in \raumL: w(t) = 0 \mbox{ a.e. in } (-b,0)\}.
\]
\begin{theorem} \label{second-order 1*}If $\varphi$, $f$, and $g$ obey Assumption \ref{A3}, then
the mapping $\tau \mapsto x[\tau]$  is twice continuously differentiable from  $[-b,0]$ to $\Lnull$.
\end{theorem}
\begin{proof} For the application of the implicit function theorem, we introduce the
mapping $F: \Lnull \times [-b,0] \to  \Lnull$ defined by the right-hand side of \eqref{integral_equation_for_x'}
by
\[
 F(w,\tau)(t) =  \int_0^t  \left\{ - (Df(x[\tau])w)(s) + A w(s-\tau) - A \dot x[\tau](s-\tau) \right\}\,ds, \ t\in [0,T],
\]
and $ F(w,\tau)(t) = 0$ for $t\in [-b,0].$

We first show that $F$ is continuously differentiable. To this end, on $[0,T]$ we split $F$ as follows:
\[
F= - \int_0^t  Df(x[\tau])w\,ds + A \int_0^t w(s-\tau)ds - A\int_0^t\dot x[\tau](s-\tau)\,ds = I_1+AI_2-AI_3.
\]
{\em Differentiability of $I_1$:} Obviously, $I_1$ is of class $C^1$ w.r. to $w \in \Lnull$. For given $w$,
the differentiability w.r. to $\tau$ is seen as follows: We have
\begin{equation}\label{I1}
Df(x[\tau])w = \sum_{i=1}^n w_i \nabla f_i(x[\tau]).
\end{equation}
For each $i$, thanks to the assumption on $f$, the mapping $x(\cdot) \mapsto \nabla f_i(x(\cdot))$ is $C^1$ in $C([0,T],\mathbb{R}^n)$. By Theorem \ref{differentiability}, the function $\tau \mapsto x[\tau]$ is $C^1$ from
$[0,b]$ to $\raumH \hookrightarrow C([-b,T];\mathbb{R}^n)$, and by the chain rule  $\tau \mapsto 
\nabla f_i(x[\tau])$ is $C^1$ from $[0,b]$ to $ C([-b,T];\mathbb{R}^n)$.

In view of \eqref{I1}, it is now easy to confirm that $(w,\tau) \mapsto  \int_0^t  (Df(x[\tau])w)(s)ds$ is of class 
$C^1$ from $\Lnull \times [0,b]$ to $L^2(0,T;\mathbb{R}^n)$.

\noindent {\em Differentiability of $I_2$:} For $w \in \Lnull$ we have
\[
\int_0^t w(s-\tau)\, ds = \int_{-\tau}^{t-\tau} w(\sigma)d\sigma = \int_0^{t-\tau} w(\sigma)d\sigma = W(t-\tau),
\]
where $W(t) =  \int_{-b}^t w(s)ds$ belongs to $\raumH$. Therefore, we have
\[
I_2(w+h,\tau+\delta)(t) = W(t-\tau-\delta)+H(t-\tau-\delta)
\]
with $H(t) =  \int_{-b}^t h(s)ds \in \raumH$. The continuous differentiability now follows from Lemma \ref{Lshift}.
The derivative in the direction $h$ is
\[
H(t-\tau) - \dot W(t-\tau) =  \int_{0}^{t-\tau} h(s)ds- w(t-\tau).
\]
Consequently, $A I_2$ is of class $C^1$ from 
$\Lnull \times [0,b]$ to $\Lnull$. 

\noindent {\em Differentiability of $I_3$:} It holds
\[
\int_0^t\dot x[\tau](s-\tau)\,ds = x[\tau](t-\tau)-x[\tau](-\tau) = \left\{
\begin{aligned}
&\varphi(t-\tau)-\varphi(-\tau),\ \quad t \in [0, \tau],\\
&x[\tau](t-\tau)-\varphi(-\tau), \ t \in (\tau,T].
\end{aligned}
\right.
\]
Both functions after  the brace belong to $H^2$ on the associated intervals. Moreover, they are equal for $t=\tau$. Thanks to Lemma \ref{piecewise*},
we are justified to perform the differentiation with respect to $\tau$ on each of the  intervals and obtain
\[
\xi[\tau](t) := \partial_\tau \int_0^t\dot x[\tau](s-\tau)\,ds =  \left\{
\begin{aligned}
&-\dot \varphi(t-\tau)+\dot \varphi(-\tau),\ \qquad \qquad \quad \ \ \, t \in [0, \tau],\\
&x'[\tau](t-\tau)-\dot x[\tau](t-\tau)+ \dot \varphi(-\tau), \ t \in (\tau,T].
\end{aligned}
\right.
\]
For all $\tau \in [0,b]$, $\xi[\tau]$ is a function of $L^2(0,T;\mathbb{R}^n)$, which depends continuously 
on $\tau$.

The differentiability properties of $I_1,\, I_2, \, I_3$ imply the continuous differentiability of $F$ and
the associated partial derivatives are the following: 

We have $(\partial_w F(w,\tau)\, z)(t) = 0$ for $t \in [-b,0]$ and
\[
(\partial_w F(w,\tau)\, z)(t)=  
\int_0^t  \left\{ - Df(x[\tau](s))z(s) + A z(s-\tau) \right\} ds, \, t \in [0,T].
\]
Moreover,  $\partial_\tau F(w,\tau)(t) = 0$ holds for $t \in [-b,0]$ and
\[
\begin{aligned}
&(\partial_\tau F(w,\tau)\, \delta)(t)=  -\delta \int_0^t D^2f(x[\tau](s))(x'[\tau](s),w(s))\,ds\\
&\hspace{4cm} + \delta A w(t-\tau) - \delta A\, \xi[\tau](t), \, t \in [0,T].
 \end{aligned}
\]
The integral equation \eqref{integral_equation_for_x'} for $w \in \Lnull$ is equivalent to
\[
w -F(w,\tau)  = 0.
\]
For all $d \in \Lnull$, the equation
\[
(I - \partial_w F(w,\tau))z = d
\]
is equivalent to a linear Volterra integral equation that can be solved stepwise in time on the intervals
$[0,\tau]$, $[\tau,2 \tau]$ etc., where the term $A z(t-\tau)$ is given from the preceding interval. 
Therefore, for all $d \in \Lnull$, the equation above has a unique solution $z$ and the mapping $d \mapsto z$
is continuous in $\Lnull$.

For all $\tau \in [0,b]$, equation \eqref{integral_equation_for_x'}  has a unique solution $w[\tau] \in \Lnull$. Thanks to the implicit function theorem, the mapping $\tau \mapsto w[\tau]$ is continuously differentiable
from $[0,b]$ to $\Lnull$. By definition, we have $w[\tau]= x'[\tau]$, hence $\tau \mapsto x'[\tau]$ is continuously differentiable. This is equivalent to the claim of the theorem. 
\end{proof}
By Theorem \ref{second-order 1*}, we are justified to differentiate equation \eqref{integral_equation_for_x'} with respect to $\tau$. This leads to the following result:
\begin{corollary}Under Assumption \ref{A3}, we obtain $x''[\tau] \in \Lnull$  as the unique solution of the integral equation
\begin{align}
&x''[\tau](t)=\int_0^t \Big\{ -\big[D^2f(x[\tau])(x'[\tau],x'[\tau]) +Df(x[\tau])x''[\tau]\big](s) + A x''[\tau](s-\tau)\Big\}ds \nonumber \\[1ex]
& \hspace{1,5cm}- 2 A x'[\tau](t-\tau) + A \dot x[\tau](t-\tau) - A \dot \varphi(-\tau),\ t \in [0,T].\label{inteq_x''}
\end{align}
\end{corollary}
Next, we derive differential equations for $x''[\tau]$. By \eqref{inteq_x''}, there holds
\begin{equation} \label{H1}
x''[\tau](t) = \int_0^t\{\ldots\} ds - 2 A x'[\tau](t-\tau) + A \dot x[\tau](t-\tau) - A \dot \varphi(-\tau), \ t \in [0,T].
\end{equation}
It follows from \eqref{H1} that  the restriction of $x''[\tau]$ to $[0,\tau]$ belongs to $H^1(0,\tau;\mathbb{R}^n)$
and the restriction of $x''[\tau]$ to $[\tau,T]$ belongs to $H^1(\tau,T;\mathbb{R}^n)$. In $t=\tau$, 
$x''[\tau](t)$ can exhibit a jump that we determine next.

For $t < \tau$, we have
\[
\lim_{t \uparrow \tau} x''[\tau](t) =  \int_0^\tau\{\ldots\} ds + A \dot \varphi(0) - A \dot \varphi(-\tau),
\]
while we find  for $t > \tau$ 
\[
\lim_{t \downarrow \tau} x''[\tau](t) =  \int_0^\tau\{\ldots\} ds + A \dot x[\tau](+0) -A \dot \varphi(-\tau).
\]
Therefore, the jump in  $t=\tau$ is
\begin{equation} \label{jump*}
x''[\tau](\tau+0)-x''[\tau](\tau-0)=  A\dot x[\tau](0^+)-A\dot \varphi(0).
\end{equation}
If the compatibility condition \eqref{compat} is fulfilled, then the jump is zero. In this case, the function
$x''[\tau]$ belongs to $H^1(-b,T;\mathbb{R}^n)$.

Two differential equations for $x''[\tau]$ can be established; one on $[0,\tau]$, another on $[\tau,T]$.

{\em Case $t \in [0,\tau]$:} Here,  the differentiation of \eqref{H1} w.r. to $t$ yields
\begin{align}
&\partial_t x''[\tau](t) =-  Df(x[\tau](t))x''[\tau](t) + Ax''[\tau](t-\tau)\nonumber \\
& \hspace{2cm}-(D^2f(x[\tau])(x'[\tau],x'[\tau]))(t) + \ddot \varphi(t-\tau), \ t \in (0,\tau],\label{x''untiltau}\\
&\quad x''[\tau](t) = 0, \ t \in [-b,0]. \label{x''until0}
\end{align}

{\em Case $t \in [\tau,T]$:} In view of \eqref{jump*}, in $t=\tau$ we have to start with the new initial value 
\[
x''[\tau](\tau+0)=x''[\tau](\tau-0)+  A\dot x[\tau](0^+)-A\dot \varphi(0).
\]
We arrive at the differential equation  
\begin{align}
&\partial_t x''[\tau](t) =-  Df(x[\tau](t))x''[\tau](t) + Ax''[\tau](t-\tau)\nonumber \\
& \qquad -(D^2f(x[\tau])(x'[\tau],x'[\tau]))(t) -2 A \dot x'[\tau](t-\tau)+  A \ddot x[\tau](t-\tau), \ t \in (\tau,T],\label{x''aftertau}\\
&\quad x''[\tau](\tau) = x''[\tau](\tau-0)+  A\dot x[\tau](0^+)-A\dot \varphi(0),\label{new_initialvalue}\\
&\quad x''[\tau](t-\tau) = x''[\tau]|_{[0,\tau]}(t-\tau), \ t \in [\tau,2\tau). \label{new_initialdata}
\end{align}
The last equation means that we have to insert $x''[\tau](t-\tau)$ obtained from the differential equation on $[0,\tau]$
in the right-hand side of \eqref{x''aftertau}.

We  differentiated \eqref{inteq_x''} on the whole interval $[\tau,T]$. Should we have expected another jump for $x''[\tau]$ of \eqref{x''aftertau}-\eqref{new_initialdata}  in $t=2\tau$? The answer is no, because
the new initial value function $\tilde \varphi(t) = x''[\tau]|_{[0,\tau]}(t-\tau)$ obeys the compatibility condition
in $t = \tau$ as can easily be checked; cf. also Remark \ref{Rsemigroup}.

Summarizing, we obtain the following information on $x''[\tau]$:
\begin{theorem} The mapping $\tau \mapsto x[\tau]$ is twice continuously differentiable with respect to $\tau$ with image in  $\Lnull$. The equation for $x''[\tau]$ is given by 
\begin{equation} \label{equation_x''_impuls2}
\begin{aligned}
&\partial_t x''[\tau](t) +Df(x[\tau](t))x''[\tau](t)  + (D^2f(x[\tau](t))(x'[\tau](t),x'[\tau](t))\\
&= Ax''[\tau] (t-\tau)  - 2A(\partial_t x'[\tau])(\fredi{t}-\tau)  + A \ddot x[\tau](t-\tau)
+ \mu_\tau \mbox{ in }(0,T],\\
&x''[\tau](t) = 0 \quad \mbox{ in }[-b,0],
\end{aligned}
\end{equation}
where $\mu_\tau= A(\dot x[\tau](0+) - \dot \varphi(0)) \, \delta(\tau)$, and  $\delta(\tau)$ is the Dirac measure located at $\tau$.
\end{theorem}
\begin{example} Continuing the discussion of Example \ref{E5.1}, we recall that\\
$
x[\tau](t)= t+1,  \ t \in (0,\tau).
$
The differential equation for $x''[\tau]$ on $[t,2\tau]$ is
\[
\dot x''[\tau] = x''[\tau](t-\tau) - 2 \dot x'[\tau](t-\tau) + \ddot x[\tau](t-\tau), \quad t \in (\tau,2 \tau).
\]
All functions on the right-hand side are zero, 
thus $x''[\tau]$ is constant on $(\tau,2\tau)$. Thanks to \eqref{new_initialvalue}, the associated initial value is 
\[
x''[\tau](\tau +0) = x''[\tau](\tau - 0) + \dot x[\tau](0+) - \dot \varphi(0),
\]
hence we find
$
x''[\tau](\tau) = \dot \varphi(-\tau)= 1. 
$
Therefore,  it holds  $x''[\tau](t) = 1$ on $[\tau,2\tau]$ and this complies with the computation of $x''[\tau]$
in Example \ref{E5.1}. 
\end{example}

\begin{example} We conclude the discussion of Example \ref{E5.1} by the  optimization problem
\[
\min_{0 \le \tau \le 1} j(\tau):= \frac{1}{2}\int_0^T |x[\tau](t) - x_d(t)|^2 dt,
\]

for equation \eqref{delay_simple} with $T=1$. We consider 3 different settings. 

(a) First,  we select $\tau = 0.5$ and $x_d=x[0.5]$, then we have $j(\tau)=0$ so that $\tau = 0.5$ affords the global
minimum to $j$. Clearly  $ j'(0.5)=0$ holds  and 
\[
j''(0.5) = \int_0^1 |x'[\tau](t)|^2 dt +\int_0^1 (x[0.5](t) - x_d(t))x''[\tau](t) dt =  \int_0^1 |x'[\tau](t)|^2 dt>0.
\]
 By Theorem \ref{theo4.1}, $\tau= 0.5$ is a strict local minimizer.

(b) Next, we fix $x_d(t) = e^t + 1$ and confirm that $\tau = 0$ is a local minimizer.  For $\tau = 0$, the delay equation reduces to the ordinary differential equation $\dot x(t) = x(t)$ with initial condition $x(0)=1$ having the solution $x[0](t)=e^t$.
Equation \eqref{equation_for_x'} for $w(t) = x'[0](t)$ becomes
\[
\dot w(t) = w(t) - \dot x[0](t) = w(t) - e^t, \quad w(0)=0
\]
with solution $-t e^t$. It follows
\[
j'(0) = \int_0^1 (x[0](t) - x_d(t))x'[\tau](t) dt =  \int_0^1 (e^t - (1+e^t))(-te^{-t})\, dt > 0.
\]
Thanks to Theorem \ref{theo4.1}, (ii), $\tau =0$ is a strict local minimizer of $j$.

(c) Finally, we select $\tau = 1$ and $x_d(t)=t$. The state associated with $\tau=1$ is $x[1](t) = t+1$.
This linear function is smaller than all other functions $x[\tau](t)$ for $\tau < 1$, hence it is the closest
to $x_d$. Notice that for $\tau < 1$ the solution $x[\tau]$ grows faster than $t+1$ for $t > \tau$. This simple observation shows that $\tau=1$ is a global minimizer of $j$. 
However, this cannot be concluded from Theorem \ref{theo4.1}, because $x'[1]=x''[1]\equiv 0$, hence $j'(1)=j''(1)=0$.
\end{example}

We conclude this section with an auxiliary result, which was used in the proof of Theorem \ref{second-order 1*}:
\begin{lemma} \label{piecewise*}If Assumption \ref{A3} is satisfied and $x[\tau]$ is the solution of \eqref{delayeq}, then 
\[
\partial_\tau \int_0^t\dot x[\tau](s-\tau)\,ds =  \left\{
\begin{aligned}
&-\dot \varphi(t-\tau)+\dot \varphi(-\tau),\ \qquad \qquad \quad  \ \ t \in [0, \tau],\\
&x'[\tau](t-\tau)-\dot x[\tau](t-\tau)+ \dot \varphi(-\tau), \ t \in (\tau,T].
\end{aligned}
\right.
\]
\end{lemma}
\begin{proof} We recall for convenience that
\[
\int_0^t\dot x[\tau](s-\tau)\,ds = \left\{
\begin{aligned}
&\varphi(t-\tau)-\varphi(-\tau),\ \ \ \quad t \in [0, \tau],\\
&x[\tau](t-\tau)-\varphi(-\tau), \quad t \in (\tau,T].
\end{aligned}
\right.
\]
The term $-\varphi(-\tau)$ appears in both intervals and does not cause difficulties for piecewise differentiation.
Therefore, it suffices to consider the function
\[
\psi(\tau,t) = \left\{
\begin{aligned}
&\varphi(t-\tau),\ \quad t \in [0, \tau],\\
&x[\tau](t-\tau), \ t \in (\tau,T].
\end{aligned}
\right.
\]
We first assume $0 < \tau < T$, then
\begin{equation} \label{partial_tau}
\partial_\tau \psi(\tau,t) := \lim_{\delta \to 0} \frac{1}{\delta}(\psi(\tau+\delta,t)-\psi(\tau,t)).
\end{equation}
For every $t \in (0,\tau) \cup (\tau,T)$, this limes exists and we obtain
\[
\partial_\tau \psi(\tau,t) =  \left\{
\begin{aligned}
&-\dot \varphi(t-\tau),\ \qquad \qquad \qquad t \in [0, \tau],\\
&x'[\tau](t-\tau)-\dot x[\tau](t-\tau), \ t \in (\tau,T].
\end{aligned}
\right.
\]
After adding the neglected term $\dot\varphi(-\tau)$, this is the claim of the Lemma in pointwise sense. We show by the Lebesgue dominated convergence theorem that the limes exists in the sense of $L^2(0,T;\mathbb{R}^n)$. For this purpose we confirm that the difference quotient above is bounded independently of $\delta$. 

We first assume $\delta > 0$
and consider the intervals
$[0,\tau], \, (\tau,\tau+\delta),\,$ and $[\tau+\delta,T]$ separately. We have
\[
\psi(\tau+\delta,t)-\psi(\tau,t) =  \left\{
\begin{aligned}
&\varphi(t-\tau-\delta) - \varphi(t-\tau),\ \qquad \quad \ \ \, t \in [0, \tau],\\
&\varphi(t-\tau-\delta) - x[\tau](t-\tau),\ \qquad \ \ t \in (\tau, \tau+\delta),\\
&x[\tau+\delta](t-\tau-\delta)-x[\tau](t-\tau), \ t \in (\tau+\delta,T].
\end{aligned}
\right.
\]
Now we derive bounds on each subinterval.

\noindent {\em Interval $[0,\tau]$:}  Since $\varphi \in H^2(-b,0;\mathbb{R}^n)$, the function $\dot \varphi$ is continuous, hence
\[
\left|\frac{1}{\delta}(\varphi(t-\tau-\delta) - \varphi(t-\tau))\right|\le \int_0^1 |\dot \varphi(t-\tau -s \delta) |ds \le \|\dot\varphi\|_{C([-b,0],\mathbb{R}^n)} < \infty.
\]

\noindent {\em Interval $[\tau+\delta,T]$:} We split
\[
\begin{aligned}
&\frac{1}{\delta}(x[\tau+\delta](t-\tau-\delta)-x[\tau](t-\tau)) = \frac{1}{\delta}(x[\tau+\delta](t-\tau-\delta)-x[\tau](t-\tau-\delta))\\
&\qquad \qquad + \frac{1}{\delta}(x[\tau](t-\tau-\delta)-x[\tau](t-\tau))) = I + II.
\end{aligned}
\]
By Corollary \ref{Cor_x'}, the function $\tau \mapsto x'[\tau]$ belongs to $C([0,b],H^1(0,T;\mathbb{R}^n)) \hookrightarrow C([0,b],C([0,T];\mathbb{R}^n))$,
hence 
\[
\begin{aligned}
&|I| \le \int_0^1\left| x'[\tau + s \delta](t-\tau-\delta)\right|ds \le \max_{(\tau,t) \in [0,b]\times [0,T]} |x'[\tau](t)|,\\
&
|II| \le \int_0^1\left| \dot x[\tau](t-\tau-s\delta)\right|ds \le \max_{t \in [0,T]} |\dot x[\tau](t)|.
\end{aligned}
\]
Here we exploited the fact that $x[\tau]\in H^2(0,T;\mathbb{R}^n)$, cf. Thm. \ref{well-posedness}.

\noindent {\em Interval $(\tau,\tau+\delta)$:} Here, the situation is a bit more difficult. By the mean value theorem
in integral form, we write
\[
\begin{aligned}
&\varphi(t-\tau-\delta) = \varphi(0) +  \int_0^1 \dot \varphi(s(t-\tau- \delta)) (t-\tau-\delta) ds\\
&x[\tau](t-\tau) = \underbrace{x[\tau](0)}_{=\varphi(0)}+  \int_0^1 \dot x[\tau](s(t-\tau)) (t-\tau) ds.
\end{aligned}
\]
Therefore,
\[
\begin{aligned}
&\frac{1}{\delta}\left|\varphi(t-\tau-\delta)-    x[\tau](t-\tau) \right| \\
&\qquad \le \int_0^1 |\dot \varphi(s(t-\tau- \delta)) |\,ds \frac{|t-\tau-\delta|}{\delta}
+ \int_0^1 |\dot x[\tau](s(t-\tau))|\,ds \frac{|t-\tau|}{\delta} \le c.
\end{aligned}
\]
Again, we invoked the $H^2$-regularity of $\varphi$ and $x[\tau]$ on $[-b,0]$ and $[0,T]$, respectively.
Moreover, we used $-\delta \le t-\tau-\delta \le 0$ and $0 \le t -\tau\le \delta$ for $t \in [\tau,\tau+\delta]$.

Thanks to our  estimates, the difference quotient $\frac{1}{\delta}(\psi(\tau+\delta,t)-\psi(\tau,t))$ is uniformly bounded for all $\delta > 0$. The case $\delta < 0$ can be handled analogously by the splitting 
$[0,T] = [0,\tau - \delta] \cup [\tau - \delta, \tau] \cup [\tau,T]$.

Now we apply the Lebesgue dominated convergence theorem for $\delta \to 0$. It implies
that the limes \eqref{partial_tau} exists in $L^1(0,T;\mathbb{R}^n)$. In view of the uniform boundedness of
the difference quotient, the limes exists in  $L^p(0,T;\mathbb{R}^n)$ for all $1 \le p < \infty$, in particular in
$L^2(0,T;\mathbb{R}^n)$.

For $\tau=0$ and $\tau = T$, we only consider the one-sided derivatives with $\delta \downarrow 0$ and
$\delta \uparrow T$, respectively, in the same way.
\end{proof}

\section{Extension to multiple time delays} \label{S6}

Here we briefly comment on the extension to an equation with multiple time delays of the form
\begin{equation} \label{delayeq_mult}
\dot x(t) + f(x(t)) = \sum_{l=1}^m A_l x(t-\tau_l) + g(t),
\end{equation}
with given matrices $A_l \in \mathbb{R}^{n\times n}$, $l=1,\ldots,m$, and delays $0 \le \tau_1 < \ldots < \tau_m \le b$. For convenience we write $\tau = (\tau_1,\ldots,\tau_m)$.

Also for multiple time delays, a discontinuity of $\dot x[\tau]$ can appear at $t=0$ only:
Indeed  the compatibility condition is now given by
$$\varphi(0) =  - f(x[\tau])(0) +  \sum_{l=1}^m A_l \varphi(-\tau_l) + g(0).$$
For $g \in H^1(0,T;\mathbb{R}^n)$, the function $\dot x[\tau]$ belongs to $H^1(0,b;\mathbb{R}^n)$. Therefore $\dot x[\tau]$ will not exhibit discontinuities after $t = 0$. However,  $t \mapsto \dot x[\tau](t)$ from $[-b,T] \to \mathbb{R}^n$ has a jump at $t=0$, in general.

Let us briefly sketch the main extensions of the results of the previous sections.

\noindent {\bf Well-posedness of \eqref{delayeq_mult} and first-order differentiability.} For the first-order analysis, we require Assumption \ref{A1}.
The well-posedness of \eqref{delayeq_mult} can be shown analogously to Theorem \ref{well-posedness}. Moreover, the first-order sensitivity analysis follows the derivation for a single time delay. 
Theorem \ref{differentiability} on existence of the first-order derivatives extends to multiple delays, i.e. to the existence of $\partial_{\tau_i} x[\tau]$, $i=1,\ldots,m$, with an analogous proof. The partial derivatives are subsequently obtained by differentiating the integral equation for $x[\tau]$ as in Corollary \ref{Cor_x'}.
We obtain that $w:= \partial_{\tau_i} x[\tau]$ is the unique solution to 
\begin{equation} \label{x'_mult}
\begin{aligned}
&\dot w(t) = - Df(x[\tau](t))w(t) + \sum_{l=1}^m A_l w(t-\tau_l) - A_i \dot x[\tau](t-\tau_i), \ t \in (0,T],\\
&w(t) = 0, \ t \in [-b,0].
\end{aligned}
\end{equation}
The adjoint equation for multiple time delays is defined by
\begin{equation} \label{ADmult}
\begin{aligned}
&-\dot p(t) = -Df(x[\tau](t))^\top p(t) +\sum_{l=1}^m  A_l^\top p(t+\tau_l) + x[\tau](t) - x_d(t), \quad t \in [0,T],\\
&\quad \ p(t) = 0, \quad t \in [T,T+b].
\end{aligned}
\end{equation}
Its unique solution $p$ is the adjoint state associated with $\tau$, denoted by $p[\tau]$.
We obtain the following results of the first and second-order sensitivity analysis of $j$:

The expression for $j'$ in terms of the adjoint is found to be 

\begin{equation} \label{Dj}
\nabla_\tau  j(\tau) = - col_i \int_0^T \langle p[\tau], A_i \dot x[\tau](t-\tau_i)\rangle dt.
\end{equation}

For the second partial derivatives of $j$  we obtain under Assumption \ref{A3}
\begin{equation}\label{D2jF}
\begin{aligned}
&\displaystyle  \partial^2_{\tau_k, \tau_i}j(\tau) =\displaystyle  \int_0^T \big\langle \frac{\partial x}{\partial{\tau_i}}[\tau], \frac{\partial x}{\partial{\tau_k}}[\tau]\big\rangle\,dt -  
\int_0^T \big\langle p[\tau],D^2f(x[\tau])(\frac{\partial x}{\partial{\tau_i}}[\tau],\frac{\partial x}{\partial{\tau_k}}[\tau]) \big\rangle dt\\[1ex]
&\qquad \displaystyle - \int_{\tau_i}^T \big \langle p[\tau](t), A_i\frac{\partial \dot x}{\partial{\tau_k}}[\tau](t-\tau_i)\big\rangle dt
- \int_{\tau_k}^T\big  \langle p[\tau](t), A_k\frac{\partial \dot x}{\partial{\tau_i}}[\tau](t-\tau_k)
\Big \rangle dt \\
&\qquad \displaystyle 
\langle p[\tau](\tau_k), A_k\big(\dot x[\tau](0^+)-\dot \varphi(0)\big)\rangle \delta_{ik}\\[1ex]
&\qquad \displaystyle +\big( \int_0^{\tau_i} \langle p[\tau](t), A_i \ddot \varphi(t-\tau_i)\rangle dt + \int_{\tau_i}^T\langle p[\tau](t),A_i \ddot x[\tau](t-\tau_i)\rangle dt \; \big) \delta_{ik},
\end{aligned}
\end{equation}
where $\delta_{ik}$ denotes the Kronecker symbol. 

\noindent {\bf Second-order differentiability of the state.} For the next results, Assumption \ref{A3} is needed. The mapping $\tau \mapsto x[\tau]$ is twice
continuously differentiable from $[0,b]^m$ to $\Lnull$.  We confirm this by the integrated version of
equation \eqref{x'_mult} for $w = \partial_{\tau_i}x[\tau]$,
\[
w(t) = \int_0^t \Big\{ - Df(x[\tau](s))w(s) + \sum_{l=1}^m  A_l w(s-\tau_l)\Big \}ds - A_i (x[\tau](t-\tau_i) - \varphi(-\tau_i)), 
\]
$t \in (0,T]$. To show  the differentiability of this equation w.r. to $\tau_j$, we apply the implicit function theorem as in the proof  of Theorem \ref{second-order 1*}.

Having the differentiability, we differentiate the integral equation above w.r. to $\tau_j$. This leads to an integral equation for $v = \partial_{\tau_j}w[\tau]=\partial_{\tau_j,\tau_i} x[\tau]$.
Taking care of possible
jumps of the functions $t \mapsto \dot x[\tau](t-\tau_i)$ and  $t \mapsto \dot x[\tau](t-\tau_j)$ in $t = \tau_i$
and $t = \tau_j$, respectively, we differentiate this equation w.r. to $t$.
Finally, we arrive at the following 
{\em delay differential equation with impulses  for $v =\partial_{\tau_k,\tau_i} x[\tau]$:}
\begin{equation}\label{eq:_ki}
\begin{aligned}
&\partial_t v(t) +Df(x[\tau](t))v(t)  + (D^2f(x[\tau](t))(\partial_{\tau_k}x[\tau](t),\partial_{\tau_i}x[\tau](t))\\
&\quad = \sum_{l=1}^m  A_l v (t-\tau)  -  A_k (\partial_{\tau_i} \dot x[\tau])(t-\tau_k) - A_i (\partial_{\tau_k} \dot x[\tau])(t-\tau_i)\\
& \qquad \quad + \delta_{ik} A_i \ddot x[\tau](t-\tau_i) + \delta_{ik} \mu_{\tau_i} \mbox{ in }(0,T],\\
&v(t) = 0 \ \mbox{ in }[-b,0],
\end{aligned}
\end{equation}
where $\mu_{\tau_i}= A_i(\dot x[\tau](0+) - \dot \varphi(0)) \, \delta(\tau_i)$.  We skip the details. 

\providecommand{\href}[2]{#2}
\providecommand{\arxiv}[1]{\href{http://arxiv.org/abs/#1}{arXiv:#1}}
\providecommand{\url}[1]{\texttt{#1}}
\providecommand{\urlprefix}{URL }

\end{document}